\documentclass[12pt,a4paper]{amsart}
\usepackage[all]{xy}
\usepackage{amsmath,amssymb,amscd,hyperref,mathtools}
\usepackage[utf8]{inputenc}
\DeclareSymbolFont{rsfscript}{OMS}{rsfs}{m}{n}
\DeclareSymbolFontAlphabet{\mathrsfs}{rsfscript}
\renewcommand{\mathcal}{\mathrsfs}

\newcommand{\DD}{{\mathbf D}}
\newcommand{\C}{{\mathbb C} }
\newcommand{\R}{{\mathbb R} }

\newcommand{\cC}{{\mathcal C} }

\newcommand{\cF}{{\mathcal F} }
\newcommand{\cI}{{\mathcal I} }
	
\newcommand{\cK}{{\mathcal K} }
\newcommand{\cL}{{\mathcal L} }
\newcommand{\cN}{{\mathcal N} }
\newcommand{\cO}{{\mathcal O} }

\newcommand{\cT}{{\mathcal T} }

\newcommand{\cX}{{\mathcal X} }
\newcommand{\cY}{{\mathcal Y} }

\newcommand{\cH}{{\mathcal H} }
\newcommand{\cB}{{\mathcal B}}

\newcommand{\wt}{\widetilde}
\newcommand{\wh}{\widehat}

\def\ol#1{{\overline{#1}}}
\def\psh{{plurisubharmonic}}

\newtheorem{theorem} {Theorem} [section]
\newtheorem{definition}[theorem] {Definition}

\newtheorem{lemma}[theorem]  {Lemma}
\newtheorem{remark}[theorem]  {Remark}
\newtheorem*{remark*}{Remark}

\newtheorem{proposition}[theorem] {Proposition}
\newtheorem{corollary}[theorem] {Corollary}

\newtheorem*{theorem*}{Theorem}
\def\uu#1{{\underline{\underline{#1}}}}

\def\ke{K{\"a}h\-ler-Ein\-stein }
\def\wp{Weil-Pe\-ters\-son }
\def\ks{Ko\-dai\-ra-Spen\-cer }
\def\ka{K{\"a}h\-ler }
\def\ii{\sqrt{-1}}
\def\ddb{\sqrt{-1}\partial\overline{\partial}}
\def\C{\mathbb{C}}

\def\cinf{C^\infty}

\def\psh{plurisubharmonic}
\def\we{\wedge}

\def\ii{\sqrt{-1}}
\def\ddb{\sqrt{-1}\partial\overline\partial}
\def\pt{\partial}

\def\backslash{\setminus}

\begin{document}

\title[The Weil-Petersson current on Douady spaces]{The Weil-Petersson current\\ on Douady spaces}

\author{Reynir Axelsson}
\address{University of Iceland, Dept.\ of Mathematics,
Dunhaga 5, \hbox{IS-107}
Reykjav\'{i}k, Iceland} \email{reynir@raunvis.hi.is}

\author{Georg Schumacher}
\address{Fachbereich Mathematik der Philipps-Universität,
Hans-Meerwein-Straße, Lahnberge, D-35032 Marburg, Germany}
\email{schumac@mathematik.uni-marburg.de}

\date{}

\begin{abstract}
The Douady space of compact subvarieties of a \ka manifold is equipped with the \wp current, which is everywhere positive with local continuous potentials, and  of class $\cinf$ when restricted to the locus of smooth fibers. There a Quillen metric is known to exist, whose Chern form is equal to the \wp form. In the algebraic case, we show that the Quillen metric can be extended to the determinant line bundle as a singular hermitian metric. On the other hand the determinant line bundle can be extended in   such a way that the Quillen metric yields a singular hermitian metric whose Chern form is equal to the \wp current. We show a general theorem comparing holomorphic line bundles equipped with singular hermitian metrics which are isomorphic over the complement of a snc divisor $B$. They differ by a line bundle arising from the divisor and a flat line bundle. The Chern forms differ by a current of integration with support in $B$ and a further current related to its normal bundle. The latter current is equal to zero in the case of Douady spaces due to a theorem of Yoshikawa on Quillen metrics for singular families over curves.

\end{abstract}
\maketitle

\section{Introduction}
The use of the \wp metric on a moduli space is the key method to apply complex analytic and geometric methods. As the Chern form of a Quillen metric on a certain determinant line bundle, the \wp form provides a link to methods of Algebraic Geometry.

In 1985 Wolpert showed \cite{wo1} that Mumford's line bundles $\lambda_{p,n}$ on a moduli space $\mathcal M_{p,n}$ of (punctured) Riemann surfaces possess a Quillen metric whose Chern form is equal to the classical \wp form. Moreover, he showed that the local potentials for the \wp forms extend to the Deligne-Mumford compactification of the moduli space in a continuous way such that the curvature current is strictly positive. He applied Richberg's argument from \cite{ri} to approximate the corresponding singular hermitian metric by a metric of class $\cinf$ with strictly positive curvature form, thus giving an analytic proof of Mumford's ampleness theorem for $\lambda$ from \cite{m}.

These results serve as a guiding principle for the analytic investigation of moduli spaces in higher dimensions. The analytic torsion for $\ol\pt$-operators over Riemann surfaces was studied by Zograf and Takhtadzhyan in \cite{zt}. The general result for determinant line bundles for proper, smooth maps and hermitian vector bundles is due to Bismut, Gillet, and Soulé \cite{bgs}.

Based upon the results from \cite{bgs} for moduli spaces of algebraically polarized extremal, compact \ka manifolds, a determinant line bundle was constructed, together with a Quillen metric, whose Chern form is the generalized \wp form \cite{f-s}. Included are moduli spaces of canonically polarized manifolds and moduli of polarized Calabi-Yau manifolds.  In \cite{sch:inv12} it was shown that a the \wp form extends as a positive current to a suitable compactification of the moduli space, and that the determinant line bundle can be extended as a holomorphic line bundle, whose curvature current is positive and extends the \wp form. Similar constructions were carried out for moduli spaces of stable vector bundles \cite{b-s2}, moduli spaces of Higgs bundles \cite{b-s-higgs},   and other cases such as Hurwitz spaces \cite{abs}.

In this article we will study Douady spaces \cite{do}. A \wp form for the space of compact submanifolds of a \ka manifold was constructed in \cite{b-s}, and further results given in \cite{a-s}. We will use Yoshikawa's theorem \cite{yo} on the singularities of Quillen metrics in an essential way.

Given a \ka manifold $(Z,\omega_Z)$ and an embedded family of subspaces $\cX \subset Z\times S$ like in \eqref{eq:doudi}, the {\em geodesic curvature} at a tangent vector $v=\pt/\pt s$ of $S$ is the pointwise $\omega_Z$-norm of the horizontal lift of $v$, and its integral over the fiber defines the \wp norm of $v$. If $\omega_\cX$ is the closed form on $\cX$ that is induced by $\omega_Z$, then the \wp form satisfies a fiber integral \eqref{eq:fibint0}, which is also meaningful when singular fibers are included. A theorem of Varouchas \cite{v1,v2} implies that the fiber integral possesses continuous, plurisubharmonic local potentials. We apply the theorem of Bismut, Gillet, and Soulé to prove that a certain determinant line bundle carries a Quillen metric, whose curvature form is equal to the \wp form.

So far this relation between Quillen metric on the determinant line bundle and \wp form only exists over the locus of smooth fibers. The determinant line bundle exists also over the whole component of the Douady space, and the \wp form exists as a current.

In \cite{sch:inv12,sch:arxiv} we showed a result about the extension of a holomorphic line bundle, equipped with a (singular) hermitian metric provided the curvature can be extended as a current. In this process the compactifying set of the underlying space may have to be blown up.

Let $\DD$ be a component of the Douady space having a non-empty intersection with the locus of smooth fibers. In our situation we have two holomorphic line bundles equipped with hermitian metrics on $\DD$: One line bundle is the determinant line bundle (called $\lambda(\xi)$) for the universal family over $\DD$, and one can see that the Quillen metric $h^Q_\xi$ defines a singular hermitian metric on $\lambda(\xi)$, whose curvature need not be positive. On the other hand we are given an extension of the determinant line bundle over the locus $\DD'\subset \DD$ of smooth fibers, based upon the \wp current. (Here a blow up $\tau:\wt\DD \to \DD$ of the locus of singular fibers in $\DD$ might be necessary --- the locus of singular fibers being replaced by a simple normal crossings divisor $B$.) These holomorphic line bundles differ by a certain factor. For such situations we prove the following result (cf.\ Theorem~\ref{th:flat}):
\begin{theorem*}
Let $Z$ be a complex manifold and $B\subset Z$ a smooth hypersurface. Let $(\cL,h)$ be a holomorphic line bundle on $Z$, whose restriction to $Z'=Z\backslash B$ is isomorphic to the trivial line bundle $(\cO_{Z'},1)$ carrying the trivial hermitian metric. Then
$$
\cL\simeq \cO_Z(\beta\cdot B) \otimes L,
$$
where $\beta$ is a rational number, and $L$ a flat line bundle.

There exists a holomorphic section $v$ of the normal bundle $N_{B|Z}$ such that the Chern current satisfies
$$
 c_1(\cL,h)= \beta[B] + T_v
$$
where $[B]$ is the current of integration and
$$
T_v(\varphi) = S_v(d\varphi)
$$
with
$$
S_v(\psi) =\int_B \mathrm{Re}(v)\,\llcorner \, \psi.
$$
In particular, the current $T_v$ with support in $B$ is exact.
\end{theorem*}
Our main application (cf.\ Theorem~\ref{th:main}) concerns the case, where $Z$ is a projective manifold:
\begin{theorem*}
Let $\lambda(\xi)$ be the determinant line bundle on $\DD$. Let $\tau:\wt\DD \to \DD$ be a modification  of the locus of singular fibers as above. Then there exists a holomorphic line bundle $L$, whose restriction to $\wt\DD\backslash B$ is trivial, together with real numbers $\wt\beta_i$ and $m\in \mathbb N$ such that
$$
\lambda^{\otimes m}_{\wt\DD}= \tau^*\lambda(\xi^{\otimes m}) \otimes \cO_{\wt\DD}\left(\sum \wt\beta_j\cdot B_i\right) L.
$$
The line bundle $L$ pulled back to a desingularization $\mu:\wh\DD\to\wt\DD$  of the logarithmic pair $(\wt\DD,B)$ is flat.
The Chern current of $(\lambda(\xi),h^Q_\xi)$ satisfies
    $$
    \omega^{WP}_{\wh\DD} = \mu^* \tau^*c_1(\lambda(\xi),h^Q_\xi)+ \sum \gamma_i[\wh B_i],\; \gamma_i\geq 0 ,
    $$
where the $[\wh B_i]$ denote the currents of integration along the components $\wh B_i$ of the snc divisor $\mu^*(B)$.
\end{theorem*}
The proof requires Yoshikawa's theorem about the singularities of Quillen metrics \cite{yo} for families over curves with smooth total space, which in our case allows logarithmic poles for $-\log h^Q_\xi$ but no simple poles along the divisor $B$ so that the current $T_v$ from the previous theorem is not present.

\section{Infinitesimal theory of Douady spaces -- Notions}
Let $(X,\cO_X)$ be a complex space and $\cF$ a coherent $\cO_X$-module. We will use the notion of a {\em small extension} of $\cO_X$ by $\cF$ in the sense of sheaves of analytic algebras. A small extension is defined by an exact sequence
\begin{equation}\label{eq:smext}
0 \to \cF \stackrel{\alpha}{\longrightarrow}\cO_Y \stackrel{\mu}{\longrightarrow} \cO_X \to 0,
\end{equation}
such that $(X,\cO_Y)$ is a complex space and
\begin{equation}\label{eq:defsmext}
a\cdot \alpha(m) = \alpha(\mu(a)\cdot m) \text{ for all } m\in \cF, a \in \cO_Y   .
\end{equation}
Equivalence classes of small extensions are defined by isomorphisms of short exact sequences \eqref{eq:smext} that are equal to the identity on $\cF$ and $\cO_X$.

For $\cF=\cO_X$ such equivalence classes are known to correspond to the space of isomorphisms classes of infinitesimal deformations i.e.\ deformations of $(X,\cO_X)$ over the double point $(\{0\},\C[t]/(t^2))$.

Let $X \subset Z$ be an embedding between (compact) complex manifolds, and $\cO_X= (\cO_Z/\cI_X)|X$. The following defines a small extension:
$$
0 \to \cI_X/\cI^2_X \to \cO_Z/\cI^2_X \to \cO_X \to 0.
$$
Then for any $\cO_X$-linear map $\psi:\cI_X/\cI^2_X \to \cO_X$ the fibered sum
\begin{equation}\label{eq:fibs}
\xymatrix{ \cI_X/\cI^2_X \ar[r]\ar[d]_\psi & \cO_Z/\cI^2_X \ar[d]\\
\cO_X \ar[r] & \cO_\cX= \cO_X \oplus_{\cI_X/\cI^2_X} (\cO_Z/\cI^2_X)  }
\end{equation}
defines a small extension of the form
\begin{equation}\label{eq:smextembx}
\xymatrix{
0  \ar[r]& \cO_X \ar[r]^\alpha&\cO_\cX \ar[r]^\mu& \cO_X \ar[r]& 0\\
& & & \cO_Z \ar[u]\ar[lu] &
}
\end{equation}
i.e.\ an infinitesimal deformation of the embedding $X\hookrightarrow Z$ with $Z$ fixed.

Conversely given a small extension of type \eqref{eq:smextembx}, the map $\cO_Z\to \cO_X$ gives rise to a morphism $\psi:\cI_X/\cI^2_X \to \cO_X$, which in turn determines $\cO_\cX$ as a fibered sum as in \eqref{eq:fibs}.

Altogether we have an identification of the space of infinitesimal deformations with $H^0(X, \cH om_{\cO_X}(\cI_X/\cI^2_X,\cO_X))$, which is equal to the space $H^0(X,\mathcal N_{X|Z})$ of holomorphic sections of the normal bundle.

In \cite{do} Douady provided the set of compact, complex analytic subspaces of a given complex space $Z$ with a natural complex structure. More precisely, he considered the functor that assigns to any complex space $S$ the set of embedded, proper, flat holomorphic families
\begin{equation}\label{eq:doudi}
\xymatrix{\cY \ar@{^{(}->}^i[r] \ar[d]_f & Z \times S \ar[dl]^{\mathrm{pr}_2} \\  S &}
\end{equation}
(where $\mathrm{pr}_2$ denotes the projection onto the second factor).
The functor assigns to any holomorphic map $R \to S$ the pullback of a given family over $R$:
$$
\xymatrix{\cY_R := \cY \times_S R \ar[d] \ar[r] & \cY \ar[d]^f\\ R \ar[r] & S}
$$
According to Douady's theorem the above functor is representable by a complex space, the {\em Douady space}.

Of particular interest is the situation, where $(Z,\omega_Z)$ is a compact \ka manifold, and $\DD_Z$ the Douady space of compact subspaces of $Z$. Here, we consider this functor on the category of reduced complex spaces, and all parameter spaces including the Douady space itself will always be a reduced complex spaces. For our later purposes it turns out that it is sufficient to study {\em normal} parameter spaces.

Given a family \eqref{eq:doudi} and a point $s_0\in S$ with smooth fiber $X=\cY_{s_0}=f^{-1}(s_0)$, the above argument defines the \ks map
$$
\rho_{s_0}: T_{s_0}S \to H^0(X,\mathcal N_{X|Z}).
$$
The \ks map can be computed as follows: Since the normal bundle $\cN_{X|\cY}$ is trivial, we get
\begin{equation}\label{eq:Normalbundles}
\xymatrix{0 \ar[r] & \cT_{X} \ar[r]\ar@{=}[d] & \cT_\cY|X
\ar[r]^{\nu_\cY}\ar[d]_{q_*}& T_{s_0}S\otimes_\C \cO_X
\ar[r]\ar[d]_\rho & 0\\0 \ar[r] & \cT_{X}  \ar[r] &\cT_Z|X
\ar[r]^{\nu_Z} & \cN_{X|Z} \ar[r]& 0}
\end{equation}
where $\cT_X$ etc.\ denote the sheaves of holomorphic vector fields, and   $q=\mathrm{pr}_1\circ i$. Now the morphism $\rho$ applied to global sections defines the \ks map $\rho_{s_0}$.

\section{Geodesic curvature and \wp form}\label{se:gecuwp}
The notion of a (generalized) \wp form for holomorphic families is closely related to the existence of distinguished \ka metrics on the fibers. The original definition by A.~Weil for compact Riemann surfaces uses the Poincaré metrics on the fibers: The corresponding $L^2$-inner product is defined in terms of harmonic \ks tensors. For families of canonically polarized manifolds the unique \ke metrics of constant Ricci curvature equal to $-1$ is taken, and again the $L^2$-inner product of harmonic \ks tensors yields a \ka metric on the parameter space.

It turned out that for (effectively parameterized) holomorphic families $f:\cX \to S$ of canonically polarized manifolds the induced hermitian metric on the relative canonical bundle $\cK_{\cX/S}$ is strictly positive with curvature form $\omega_\cX$ such that for all fibers $\cX_s=f^{-1}(s)$, $s\in S$ the restrictions $\omega_\cX|\cX_s$ are the original \ke metrics on the fiber \cite{sch:inv12}.

We note that by definition a \ka form on a reduced complex space has local $\ol\pt\pt$-potentials of class $\cinf$ given locally on a smooth ambient space. By definition it is positive definite on all Zariski tangent spaces. The pair $(\cX,\omega_\cX)$ will be called a {\it \ka space}.

In such a situation the notion of the (generalized) {\em geodesic curvature} is meaningful. We will describe the general approach.

Let $f:\cX \to S$ be a proper, smooth, holomorphic map of reduced spaces, and $\omega_\cX$ a \ka form.

Let $s^i$ be coordinates on a locally given smooth ambient space of $S$ near a given point $s_0\in S$, and $z^\alpha$, $\alpha=1,\ldots,n$ coordinates on $\cX_{s_0}$ such that $(z,s)$ serve as local coordinates on $\cX$ with $f(z,s)=s $. Let
$$
\omega_\cX = \ii\left(g_{i\ol\jmath}ds^i\we ds^{\ol\jmath} + g_{\alpha\ol\jmath}dz^\alpha\we ds^{\ol\jmath} + g_{i\ol\beta} ds^i\we dz^{\ol\beta} + g_{\alpha\ol\beta}dz^\alpha\we dz^{\ol\beta}  \right)
$$
\begin{definition}\label{de:canli}
Let $\pt/\pt s^i|_{s_0}\in T_{s_0}S$ be a tangent vector. Let
$$
v_i = \frac{\pt}{\pt{s^i}} + a^\alpha_i \frac{\pt}{\pt z^\alpha}
$$
be the differentiable vector field on $\cX_{s_0}$ with values in $\cT_\cX$ that is perpendicular to the fiber $\cX_{s_0}$ with $f_*(v_i) = \pt/\pt s^i$. Then $v_i$ is called the horizontal lift of $\pt/\pt s^i$.
\end{definition}
Note that a simple calculation shows
\begin{equation}\label{eq:aalphai}
g_{\alpha\ol\beta} a^\alpha_i = - g_{i\ol\beta}.
\end{equation}
\begin{definition}\label{se:geocurv}
The pointwise inner product of horizontal lifts with respect to $\omega_\cX$
$$
\varphi_{i\ol\jmath} = \langle v_i,\\v_j\rangle_{\omega_\cX}
$$
is called geodesic curvature.
\end{definition}
The following equation holds
$$
\varphi_{i\ol\jmath} = g_{i\ol\jmath} - g_{\alpha\ol\beta}a^\alpha_i a^{\ol\beta}_{\ol\jmath}.
$$
\begin{remark}
For $s_0\in S$ the geodesic curvature is a  family of hermitian inner products on $T_{s_0}S$ depending in a $\cinf$ manner on the points of the fiber $\cX_{s_0}$.
\end{remark}

\begin{remark}
	Both the definition of a horizontal lift of a tangent vector, and the definition of the geodesic curvature, are meaningful on the smooth locus of a fiber $\cX_s$.
\end{remark}
Note that for families of compact (punctured) Riemann surfaces and canonically polarized varieties the \wp form
$$
\omega^{WP}= \ii G^{WP}_{i\ol\jmath}(s)\; ds^i\we ds^{\ol\jmath}
$$
can be computed in terms of the geodesic curvature (cf.\ \cite{sch:inv12}):
$$
G^{WP}_{i\ol\jmath} = \int_{\cX_s} \varphi_{i\ol\jmath}(z,s)\; \omega^n_{\cX_s},\qquad n=\dim{\cX_s}.
$$
For differential forms $\eta$, and $m\in \mathbb N$ we set $\eta^m=(1/m!)\eta\we\ldots\we \eta$.

We have
$$
\omega^{n+1}_\cX =  \varphi_{i\ol\jmath}\ii ds^i \we ds^{\ol\jmath}\we \omega^n_\cX,
$$
up to contributions whose integrals along fibers of $f$ vanish.

\begin{proposition}[\cite{sch:inv12,b-s,a-s}]\label{pr:fibint}
Let $f:\cX \to S$ b e a proper, smooth holomorphic map with fiber dimension $n$, and $\omega_\cX$ be any \ka form $\omega_\cX$. Define $\omega^{WP}$ in terms of the geodesic curvature as above. Then a fiber integral formula
\begin{equation}\label{eq:fibint0}
  \omega^{WP}= \int_{\cX/S} \omega^{n+1}_{\cX}
\end{equation}
holds.
In particular $\omega^{WP}$ is a closed form on $S$.
\end{proposition}
Note that the above formula holds for (reduced) singular base spaces. In fact, only the first infinitesimal neighborhood of a point in $S$ matters for the proof.

A fiber integral like \eqref{eq:fibint0} already appeared in the algebraic context of \cite{an}. Later it appeared in \cite{v1,v2,f-s}, whereas  in \cite{campana,f1} a more general situation was considered.

In \cite{b-s,a-s} we introduced the approach via geodesic curvatures to Douady spaces, more precisely to the open subspaces with smooth fibers.

In a sequence of articles \cite{f0,f1,f2}, Fujiki studied Douady and Barlet spaces. We will need the following version of {\cite[Theorem 5.3]{f1}}:
\begin{theorem}[Fujiki]
Let $(Z,\omega_Z)$ be a compact \ka manifold. Then all irreducible components $\DD_\alpha$ of the Douady space $\DD_Z$ are compact.
\end{theorem}
Let $(Z,\omega_Z)$ be a compact \ka manifold, and $\DD\subset \DD_Z$ an irreducible component of the (reduced) Douady space, such that the locus $\mathbf{D'}\subset \DD$ of points with smooth fibers is non-empty. Denote by $n$ the dimension of the fibers $\cX_s$ for $s\in \DD$. We consider the universal family
\begin{equation}\label{eq:embfam}
\xymatrix{\cX'  \ar@{^{(}->}^{i'}[r] \ar[d]_{f'} & \cX \ar@{^{(}->}^i[r] \ar[d]_f & Z \times \DD \ar[dl]^{\mathrm{pr}_2} \\ \mathbf{D'} \ar@{^{(}->}[r] & \DD &}
\end{equation}
Set $\omega_\cX= i^* \mathrm{pr}_1^*\omega_Z$, where $\mathrm{pr}_1: Z\times \DD \to Z$ is the projection. The forms $\omega_\cX$ and $\omega_\cX^{n+1}$ locally are restrictions of $\cinf$ forms on smooth ambient spaces. With this setting Proposition~\ref{pr:fibint} holds for the \wp form $\omega^{WP}_{\mathbf{D'}}$ on $\mathbf{D'}$:

The form $\omega_\cX$ need not be strictly positive in all directions, but only semi-positive. However, since all fibers of $f$ are subspaces of $Z$, the form $\omega_\cX$  is strictly positive when restricted to fibers. This fact guarantees the  existence of horizontal lifts of tangent vectors at regular points of the fibers by \eqref{eq:aalphai}. However, the pointwise $\omega_\cX$-norms of horizontal lifts need not be bounded.  Now, if a $\omega^{WP}$\!-degenerate tangent vector with horizontal lift $v$ existed (over the regular locus of the fibers), the corresponding geodesic curvature would have to vanish identically at all smooth points of the fiber, because it has to be semi-positive everywhere. It follows from the construction that the projection $\mathrm{pr}_{1*}i_*: T_{\cX_s,x} \to T_sS $ applied to $v$ would vanish for all points  $x\in \cX_s$. Hence $i_*v(x)$ would be horizontal in $Z\times D$, which means that the given tangent vector would have a holomorphic lift, and the deformation would be infinitesimally trivial in the given direction. The above construction is compatible with base change. Altogether we note that the \wp form is strictly positive in the sense of currents at all points, whose fibers are reduced (or have at least one reduced irreducible component).

\section{The \wp current}\label{se:WPCurr}
We consider the universal family \eqref{eq:embfam} over an irreducible component $\DD$ of the Douady space $\DD_Z$.

In \cite[Theorem~2.4.1]{a-s} we extended the \wp form given on the space given on the smooth locus $\mathbf{D'}$ of $f$ to the whole of  $\DD$, and showed that the extended form $\omega^{WP}_\DD$ possesses locally plurisubharmonic, continuous $\pt\ol\pt$-potentials. We indicate briefly our argument, which is based upon a theorem by Varouchas.
\begin{theorem}[{\cite[Theorem 2]{v2}}]\label{th:VAthm2}
For any $s\in \DD$ there exists an open neighborhood $S\ni s$ and a differential form $\psi$ of degree $(n,n)$ on $f^{-1}(S)$ of class $\cinf$ (being the restriction of a differentiable form on a smooth ambient space) such that
$
\ddb \psi = \omega^{n+1}_\cX|f^{-1}(S).
$
\end{theorem}

We will be interested in the fiber integral/push forward of such an $(n,n)$-form $\psi$.  In the case that $Z$ is a projective space  Andreotti and Norguet showed in \cite{an} that the fiber integral of $\psi$ defines a continuous  function $\chi(s)$. A  more general case was treated by Stoll in \cite{st} (cf.\ also the results in \cite{f1}).

Our aim is to interpret the fiber integral
\begin{equation}\label{eq:fibint}
 \int_{\cX/\DD} \omega_\cX^{n+1}
\end{equation}
as \wp current $\omega^{WP}=\omega^{WP}_\DD$. By the above theorem it has a plurisubharmonic potential and hence can be pulled back under proper holomorphic maps.
\begin{proposition}[cf.\ {\cite[Theorem~2.4]{a-s}}]\label{pr:contpot}
The fiber integral \eqref{eq:fibint} defines a positive current with local, continuous potentials on each component $\DD$ of $\DD_Z$. The current will be called \wp current and denoted by $\omega^{WP}$. The construction is functorial, i.e.\ compatible with base change as long as the general fiber is smooth.
\end{proposition}
\begin{proof}
We know that the push forward of the current $\omega_\cX^{n+1}$ is positive in the sense of currents.

The claim being local with respect to the vase space $\DD$, we can replace $\DD$ by an open subset $U\subset \DD$, where a \ka form $\omega_U$ exists. We set ${\cX_U}= f^{-1}(U)$. With the notion of \eqref{eq:embfam} the form $\mathrm{pr}^*_1\omega_{Z}+ \mathrm{pr_2}^*\omega_U$ is a \ka form on $Z\times U$ so that $\omega_{\cX}|\cX_U+ f^* \omega_U$ is a \ka form on $f^{-1}(U)$. Now Theorem~\ref{eq:embfam} implies that the following current possesses continuous local $\ddb$-potentials.
\begin{eqnarray*}
\int_{\cX_U/U} (\omega_{\cX|U} + f^*\omega_U)^{n+1} &=& \int_{\cX_U/U} \omega_{\cX|U}^{n+1} + \int_{\cX_U/U} (\omega_{\cX|U}^{n} \we f^*\omega_U)\\& = & \omega^{WP}_U + \mathrm{vol}(\cX_s) \, \omega_U.
\end{eqnarray*}
The latter equality holds because of the existence of continuous $\ddb$-potentials. This shows the claim. The construction of the \wp current is clearly functorial in the stated sense -- here we need the fact that continuous local potentials for $\omega^{WP}$ exist so that pull-backs of the \wp current exist as long as points with smooth fibers are present.
\end{proof}

\section{Determinant Line Bundle}
\subsection{BGS-Theorem and Bott-Chern form}
For a proper,  holomorphic map $f: \cX \to S$ of complex spaces we use the notation $\uu{R} f_* : K_0^{\text{hol}}(\cX) \to K_0^{\text{hol}}(S)$ for the direct image functor in the derived category (extended to the Grothendieck group). We assume that $f$ is flat, and that $\xi$ is a locally free $\cO_X$ module. Then the Knudsen-Mumford determinant $(\det\uu{R} f_* \xi)^{-1}$ commutes with base change (\cite[p.\ 46]{k-m}). We denote the Knudsen-Mumford invertible sheaf by
$$
\lambda(\xi)= \left(\det\uu{R} f_* \xi\right)^{-1}.
$$
We will need the Quillen metric on determinant line bundles.  According to Bismut, Gillet, and Soul\'e \cite{bgs}, the Grothen\-dieck-Hir\-ze\-bruch-Riemann-Roch theorem holds, in the case of a proper, {\em smooth} family $f: \cX \to S$ over a smooth base space $S$, for distinguished differential forms in degree $2$, rather than cohomology classes, where $\cX$ carries a relative \ka structure $\eta_{\cX/S}$,  and $(\xi,h)$ a hermitian vector bundle. We will see that in our situation $\eta_{\cX/S}$ is only of auxiliary nature.
\begin{theorem}[{\cite{bgs}}]
Denote by ${\rm td}$ and ${\rm ch}$ resp.\ the Todd and Chern character forms. Then
\begin{equation}\label{eq:BGS}
c_1(\lambda(\xi),h^Q)=- \left[\int_{\cX/S} {\rm td}(\cX/S,\eta_{\cX/S})\cdot
{\rm ch}(\xi,h)\right]^{(1,1)}.
\end{equation}
\end{theorem}
The statement of the theorem was verified also in the case where $S$ denotes a reduced complex space in \cite[Section 12]{f-s}.

In \cite{a-s} we used the results by Bismut \cite{bi1,bi2} about the generalization of \eqref{eq:BGS} to singular families, together with the formulas from \cite{ott1,ott2}.

In the above situation, let $\wt\eta_{\cX/S}$ be a further relative \ka form. Denote by $h^Q$, and $\wt h^Q$ the respective Quillen metrics on $\lambda(\xi)$. Then (in our notation) the corresponding Bott-Chern form is denoted by $\wt{Td}(\eta_{\cX/S},\wt\eta_{\cX/S})$:
\begin{equation}
\frac{\ii}{2\pi}\, \pt \ol{\pt}\, \wt{\mathrm{td}}(\eta_{\cX/S},\wt\eta_{\cX/S})= \mathrm{td}(\eta_{\cX/S})- \mathrm{td}(\wt\eta_{\cX/S})
\end{equation}
\begin{theorem}[{\cite[(0.4)]{bgs}}]\label{th:bgs_cw}
	When the relative \ka form $\eta_{\cX/S}$ is replaced by $\wt\eta_{\cX/S}$ the Quillen metric is multiplied by the exponential of the component of degree zero of the following integral:
	\begin{equation}\label{eq:wthq}
    \wt h^Q  =  \exp \left[ \int_{\cX/S} \wt{\mathrm{td}}(\eta_{\cX/S},\wt\eta_{\cX/S}) \cdot \mathrm{ch}(\xi,h) \right]^{(0)}\!\!\! \cdot h^Q
\end{equation}
\end{theorem}

A calculation of the Quillen metric in a degenerate situation was achieved by Yoshikawa \cite{yo}. We will describe his result briefly.

\subsection{Barlet class}
The Barlet class $\mathcal B(U)$ of continuous functions of one complex variable is defined as follows: Let $0\in U\subset \C$, then $\varphi(t)\in \mathcal B(U)$, if
$$
\varphi(t)=f_0(t) + \sum^{m}_{\ell=1}\sum^n_{k=0}|t|^{2r_i}(-\log|t|)^k f_{\ell,k(t)}
$$
for some $r_1,\ldots,r_m \in \mathbb Q\cap(0,1]$ and $f_0, f_{\ell,k}\in \cinf(U)$ for $\ell=1,\ldots m$, $k=1,\ldots,n$. We denote functions equivalent up to the Barlet class ($\equiv_{\cB}$), if their difference is contained in $\cB(U)$.

\subsection{Yoshikawa's result}
Yoshikawa considers the following situation; we keep his notation: Let $X$ be a compact (connected) \ka manifold of dimension $n+1$ with \ka metric $g_X$, $S$ a compact Riemann surface, and $\pi: X \to S$ a surjective holomorphic map.  Let $U$ be the universal hyperplane bundle of rank $n = \dim X/S$ over $\mathbb P(TX)^\vee$, and let $H := \cO_{\mathbb P(TX)^\vee}(1)$. The map $q:\wt X \to X$ denotes a resolution of singularities of the Gauss map $\mu$ (cf.~\cite{yo} for details). The Gauss map is holomorphic over the regular locus of $\pi$ so that $q$ is a modification, whose exceptional locus $E\subset \wt X$ is a simple normal crossings divisor, which is mapped under $q$ to the singular fibers of $\pi$. The Gauss map with singularities resolved is denoted by $\wt\mu$.

\begin{theorem}[Yoshikawa {\cite[Theorem 1.1]{yo}}]\label{th:yo}
Let $t$ denote a local holomorphic coordinate near a singular value of $\pi$ such that the singular point corresponds to the value $t=0$ with $X_0=\pi^{-1}(0)$. Let $(\xi,h)$ be a (virtual) hermitian holomorphic vector bundle.

Up to an element of the Barlet class the following identity holds
\begin{equation}\label{eq:yo}
 \log\|\sigma\|^2_{Q,\lambda(\xi)} \equiv_B \left( \int_{E\cap q^{-1}(X_0)} \wt\mu^*\left\{\mathrm{td}(U)\frac{\mathrm{td}(H)-1}{c_1(H)}  \right\} q^*\mathrm{ch}(\xi)    \right)\log|t|^2
\end{equation}
where $\sigma$ is a local section of the determinant line bundle given by a local coordinate function $t$ vanishing at the singular point of the map.
\end{theorem}

\section{Application of the BGS-Theorem}\label{se:appBGS}
Let $Z\subset \mathbb P_N$ be a projective manifold, equipped with a positive line bundle $(\cL_Z, h_Z)$  such that $\omega_Z= c_1(\cL_Z,h_Z))$ is a \ka form. We may assume that $(\cL_Z, h_Z)$ is the line bundle $\cO_Z(1)$ equipped with the Fubini-Study hermitian metric.

We consider the family \eqref{eq:embfam} of embedded subvarieties of dimension $n$ over an irreducible component $\DD\subset \DD_Z$. We define $\omega_\cX= \omega_Z|\cX$ as above inducing a relative \ka structure.

We denote the restriction of $(\mathrm{pr}^*_1\cL_Z, \mathrm{pr}^*_1 h_Z)$ to $\cX$ by $(\cL_\cX, h_\cX)$, and we denote by $\xi$ the element $(\cL_\cX - \cO_\cX)^{\otimes(n+1)}$ of the Grothendieck group, which carries an induced hermitian metric (as element of the Grothendieck group).

We use the earlier argument showing that
\begin{equation}\label{eq:xi}
\mathrm{ch}(\xi,h)= c_1(\cL_\cX,h_\cX)^{n+1} + \ldots = \omega_{\cX}^{n+1}+ \ldots
\end{equation}
so that the term in degree $(n+1,n+1)$ of the integrand in \eqref{eq:BGS} equals $\omega_\cX^{n+1}$.

Let $(\xi,h)$ be a hermitian virtual vector bundle of the above type.

Now the following problem arises when applying the BGS-Theorem to the hermitian virtual vector bundle $(\xi,h)$: For our special choice that satisfies \eqref{eq:xi} the integrand in   \eqref{eq:wthq} does not depend upon the relative \ka form that is used to compute the relative Todd character. only the term of degree zero, namely the number $1$ matters. However, this fact does not imply right away that the corresponding Quillen metric on the determinant line bundle is independent of the choice of the auxiliary relative \ka form $\eta_{\cX/S}$ for the family $\cX \to S$.

\begin{proposition}\label{pr:indep}
Let $(\xi,h)$  be a hermitian virtual vector bundle of degree zero. Then
$$
c_1(\lambda(\xi),h^Q)= \left[\int_{\cX/S} {\rm ch}(\xi,h)\right]^{(1,1)}.
$$
Moreover, the Quillen metric $h^Q$ does not depend upon the choice of the relative \ka form $\omega_{\cX/S}$.
\end{proposition}
\begin{proof}
The first statement follows from \eqref{eq:xi}, and Section~3.
For the second statement we refer to Theorem~\ref{th:bgs_cw}: Again the term of lowest degree in the integrand of  \eqref{eq:wthq} is of degree $(n+1,n+1)$ so that the degree zero term of the fiber integral \eqref{eq:wthq} is identically zero.
\end{proof}

\begin{proposition}
\strut
\begin{enumerate}
\item[(i)]
The Knudsen-Mumford determinant line bundle
$$
\lambda^{KM}(\xi)= \det\uu{R} f_*\left((\cL_\cX - \cO_\cX)^{\otimes(n+1)}\right)^{-1}
$$
is defined over the irreducible component $\DD$ of the Douady space.
\item[(ii)]
Over the locus $\mathbf{D'}$ of embedded submanifolds
\begin{equation}\label{eq:wpdetbdl}
c_1(\lambda(\xi),h^Q)= \omega^{WP} \text{ for } \lambda(\xi)=\lambda^{KM}(\xi)^{-1}
\end{equation}
holds.
\item[(iii)]
The \wp current is defined over $\DD$ and possesses locally  plurisubharmonic, continuous potentials. Near points with reduced fibers the potential can be chosen as strictly plurisubharmonic.
\end{enumerate}
\end{proposition}
\begin{proof}
The first part was shown above. The second part follows from Proposition~\ref{pr:fibint}, \eqref{eq:fibint0}, \eqref{eq:BGS}, and \eqref{eq:xi},  and the last statement follows from the discussion in the preceding paragraph. We saw that the \wp norm of a tangent vector may be equal to $+\infty$ at points with reduces singular fibers. At points with non-reduced fibers strict positivity cannot be expected.
\end{proof}

\section{Application of Yoshikawa's Theorem}\label{se:appY}
\subsection{Smooth base curves}
Let $(\cX,\wt\omega_\cX)$ be a compact smooth \ka space, $C$ a compact Riemann surface, and $f:\cX\to C$ a proper, flat, holomorphic map that is smooth over the complement $C'\subset C$ of a discrete set. Let $(\xi, h_\xi)$ be a holomorphic, hermitian (virtual) vector bundle on $\cX$ with determinant line bundle $\lambda(\xi)$. Over $C'$ we have a Quillen metric $h^Q_\xi$. If $f$ is the restriction of a projection map  $\mathrm{pr}:Z\times C \to C$ to $\cX\subset Z\times C$, we consider a desingularization $\mu:\cY \to \cX$ (by a sequence of blow-ups of $Z$ with smooth centers, and restriction to the proper transform of $\cX$), and the diagram
\begin{equation}\label{eq:famC}
\xymatrix{
\cY \ar[r]^\mu\ar[dr]_{\wt f} & \cX \ar[d]^f  \ar@{^{(}->}^i[r]  &  Z \times C \ar[dl]^{\mathrm{pr}}\\ & C &
}
\end{equation}
Since $C$ is of dimension one and smooth, the map $\wt f$ is flat. We establish the situation of Section~\ref{se:appBGS} and assume that $Z\subset \mathbb P_N$, and $\omega_Z= c_1(\cL_Z,h_Z)$.

Note that allowing for non-reduced fibers we do not need the semi-stable reduction theorem. The space $\cY$ is known to be K\"ah\-ler, we chose an auxiliary \ka form $\omega_\cY$. Concerning the notation of determinant line bundles, we indicate the respective map as a subscript, if necessary.

Yoshikawa's Theorem \cite[Theorem 1.1]{yo} is applicable to the (flat) family $\wt f:\cY \to C$ together with the determinant line bundle $\lambda_{\wt f}(\mu^*\xi)$ and the Quillen metric $h^Q_{\mu^*\xi}$ over $C'$. It follows from \eqref{eq:xi} that the contribution of the integral \eqref{eq:yo}vanishes, because the integrand in degree $(n,n)$ is identically zero. Hence Theorem~\ref{th:yo} immediately implies the following statement.

\begin{proposition}\label{pr:cont}
  The logarithm of the Quillen metric
  $$
  \log(h^Q_{\mu^*\xi})
  $$
  is continuous on all of $C$.
\end{proposition}

By Proposition~\ref{pr:indep} the metric $h^Q_{\mu^*\xi}$ is independent of the choice of the \ka metric on $\cY$. Over $C'$ the line bundles $\lambda_f(\xi)$, and $\lambda_{\wt f}(\mu^*\xi)$ are isomorphic. Hence $h^Q_\xi$ corresponds to $h^Q_{\mu^*\xi}$ with respect to such an isomorphism over $C'$. More can be said.

\begin{proposition}\label{pr:comphQ}
The canonical morphism $\uu R f_* \xi \to \uu R\wt f_*(\mu^* \xi)$  induces a natural morphism $\lambda_f(\xi) \to \lambda_{\wt f}(\mu^*\xi)$ so that we can identify $\lambda_f(\xi)$ with $\lambda_{\wt f}(\mu^*\xi)(-D)$ for some effective divisor $D$ on $C$. Let $\sigma$ be a canonical section of $\cO_C(D)$. Then
\begin{equation}\label{eq:comphQ}
h^Q_{\mu^*\xi}\cdot |\sigma|^2= h^Q_\xi
\end{equation}
holds over $C'$ with respect to the above identification of determinant line bundles.
\end{proposition}
Note that for $\lambda_f(\xi)$ over $C$ a Quillen metric is not defined. However, we can use \eqref{eq:comphQ} to define such an extension $h^Q_\xi$ to all of $C$, which is a singular hermitian metric on  $C$. Its curvature might not be a positive current unless $D=0$.

Since the \wp current possesses local continuous potentials, restrictions to subspaces that intersect the regular locus of the restricted  family are well defined. We now can identify a local potential for the \wp current by Proposition~\ref{pr:cont}, and Proposition~\ref{pr:comphQ} according to Theorem~\ref{th:yo}. It turns out that the \wp current, as it was defined by the fiber integral \eqref{eq:fibint}, can be computed from the pullback of $\xi$ to a desingularization and the continuous extension of the Quillen metric in the sense of Yoshikawa.

This means that the \wp form is determined by the Quillen metric taken for the desingularization.

(Observe that the process is independent of the choice of the desingularization, since further blow-ups with smooth centers of non-singular spaces do not change the cohomology, and since \eqref{eq:wthq} together with \eqref{eq:yo} imply that also the Quillen metrics are invariant under this further process.)
\begin{theorem}\label{th:wpC}
Let $\omega^{WP}_C$ be the \wp current for a family \eqref{eq:famC}. Then
$$
\omega^{WP}_C = c_1(\lambda_{\wt f}(\mu^*\xi), h^Q_{\mu^*\xi}).
$$
\end{theorem}

\begin{proof}
We have
$$
\omega^{WP}_C = \int_{\cX/C} \omega^{n+1}_\cX =     \int_{\cY/C} (\mu^*\omega_\cX)^{n+1} = c_1(\lambda_{\wt f}(\mu^*\xi, h^Q_{\mu^*\xi})),
$$
where the first equality holds by definition \eqref{eq:fibint}. The second equality holds, because the fibers of $\mu$ are connected, and the third equality is a consequence of Proposition~\ref{pr:indep}, Theorem~\ref{th:yo}, and Theorem~\ref{th:VAthm2}.
\end{proof}
In particular, we have seen that $\log(h^Q_{\mu^*\xi})$ can serve as a continuous potential for $\omega^{WP}_C$. Combining this with Proposition~\ref{pr:comphQ} we get the following immediate consequence.
\begin{corollary}\label{co:wp}
  The \wp current is equal to
  \begin{equation}\label{eq:owpC}
  \omega^{WP}_C = c_1\left(\lambda_f(\xi)\otimes \cO_C(D),\frac{1}{|\sigma|^2}h^Q_\xi\right) = c_1(\lambda_f(\xi), h^Q_\xi) + [D]
  \end{equation}
where $[D]$ denotes the current of integration along $D$ with (positive, integer) multiplicities.
\end{corollary}

\subsection{Singular base curves}
We now discuss the situation, where the base curve $C$ is allowed to be singular.

The statement of this section remains valid, if $C$ is a singular, compact, locally irreducible, complex curve in the following sense.

Let $\nu:\wh C \to C$ be the normalization, and $\wh \cX = \cX \times_C \wh C$. As above we take a desingularization $\mu:\wh \cY \to \wh \cX$:
\vspace{-2mm}
$$
\xymatrix{
\wh \cY \ar[r]^\mu\ar[dr]_{\wt f} & \wh\cX \ar[d]^{\wh f} \ar[r]^{\wh\nu}  & \cX \ar[d]^f  \\ & \wh C \ar[r]^\nu & C.
}
$$
We use that fact that according to \cite[Section 12]{f-s} the BGS-Theorem is valid on the locus where the given family is smooth. The base change theorem provides us with a morphism $\nu^* \uu R f_* \xi \to \uu R \wh f _*(\wh\nu^* \xi)$, in particular we get a morphism $\nu^* \lambda_f(\xi) \to \lambda_{\wh f}(\wh\nu^* \xi)$, and like in Proposition~\ref{pr:comphQ} there is a morphism $\lambda_{\wh f}(\wh\nu^*\xi) \to \lambda_{\wt f}(\mu^*\nu^*\xi)$.

Let $C'\subset C$ be the locus of non-singular fibers of $f$. We know from \cite{f-s} that \eqref{eq:BGS} holds also at the singularities of $C'$. Hence there is an effective divisor $D'$ supported on $\nu^{-1}(C\backslash C')$, and a canonical section $\sigma'$ of $\cO_{\bf \wh C}(D')$ such that
\begin{equation}\label{eq:comphQg}
h^Q_{\mu^*\wh\nu^*\xi} \cdot|\sigma'|^2 = h^Q_{\wh\nu^*\xi}\, .
\end{equation}
This equation implies the existence of a singular hermitian metric $h^Q_\xi$ with $\nu^*h^Q_{\xi}   =h^Q_{\wh\nu^*\xi}$.

Again, because of the existence of a continuous potential for a \wp metric the pullback $\nu^*\omega^{WP}_{C}$ is well-defined. We now are able to include the locus of singular fibers of $f$, Hence
\begin{equation}\label{eq:wpchat}
 \omega^{WP}_{\wh C} =  \nu^* \omega^{WP}_{C} = \nu^*c_1(\lambda_f(\xi), h^Q_{\xi}) + [D'].
\end{equation}
One can see that \eqref{eq:wpchat} yields an analogous statement to Corollary~\ref{co:wp}:
$$
\omega^{WP}_C =  c_1(\lambda(\xi), h^Q_\xi) + [D],
$$
where the current of integration $[D]= \sum m_j [p_j]$, $m_j\in \mathbb N$, in \eqref{eq:owpC} has to be replaced by a similar current with $m_j\in \mathbb Q$, $m_j>0$. Again, $h^Q_\xi$ is a singular hermitian metric (which is not necessarily of positive curvature): In order to see this, we take a high order multiple $mD$, which descends to the curve $C$.

\section{Singular hermitian metrics and residues}
\begin{proposition}[{\cite{sch:inv12,sch:arxiv}}]\label{pr:ext}
Let $Z$ be a complex manifold and $B\subset Z$ a smooth, connected divisor. Let $Z'=Z\backslash B$. Let $(\cL',h')$ be a holomorphic line bundle together with a singular hermitian metric over $Z'$ of positive curvature in the sense of currents. Then $(\cL',h')$ can be extended to $Z$ as holomorphic line bundle $(\cL,h)$ equipped with a singular hermitian metric (of positive curvature) provided the Chern form $c_1(\cL',h')$ extends to $Z$ as a positive current $\omega$. In this case
\begin{equation}\label{eq:extmf}
  c_1(\cL,h)=\omega + \alpha[B]
\end{equation}
where $[B]$ denotes the current of integration along $B$, and $0\leq \alpha <1$.
\end{proposition}

The general version of the result is the following.
\begin{theorem}[{\cite{sch:inv12}, cf.\ \cite[Theorem~1]{sch:arxiv}}]\label{th:mainext}
Let $Y$ be a reduced complex space and $A\subset Y$ a closed analytic subset. Let $(L',h')$ be a holomorphic line bundle on $Y'= Y\backslash A$ equipped with a singular hermitian metric whose Chern current $c_1(L',h')$ is positive. We assume that there exists a desingularization of $Y$ such that the pullback of $\omega_{Y'}$ extends as a positive, closed $(1,1)$-current.

Then there exists a modification $\tau: \wt Y\to Y$, which is an isomorphism over $Y'$ such that $(L',h')$ extends to $\wt Y$ as a holomorphic line bundle equipped with a singular hermitian metric, whose curvature form is a positive current.
\end{theorem}

In the situation of the Theorem let $B=\tau^{-1}(A)$; this set can be chosen as a simple normal crossings divisor with components $B_i$.  In order to relate the statement to Proposition~\ref{pr:ext} a desingularization $\mu: Z \to \wt Y$ of the logarithmic pair $(\wt Y, B)$ is necessary, and the pull-back of the line bundle to $Z'$.

Note that an {\em extension} of a holomorphic line bundle $\cL'$ from $Z'$ to $Z$ is a holomorphic line bundle $\cL$, whose restriction to $Z'$ is isomorphic to $\cL'$, where the isomorphism is taken in the holomorphic category, and may involve transcendental holomorphic functions.

In this section we will discuss the problem to what extent such an extension is unique.

\medskip

\noindent {\bf Setup.}
\textit{Let $(\cL,h)$ be a holomorphic line bundle over a complex manifold $Z$ equipped with a singular hermitian metric, whose Chern form is trivial when restricted to $Z'=Z\backslash B$. Here $B$ denotes a connected, smooth divisor, and later a simple normal crossings divisor.}

We first treat the local, smooth situation: Let $\Delta\subset\C$ and $\Delta^*\subset\C$ be the unit disk, and the punctured unit disk resp. Let $B=\{0\}\times\Delta^{n-1}$, and let $[B]$ be the current of integration along $B$. On $\Delta^*\times \Delta^{n-1}$ and $\Delta^n$ we use coordinates $z=(z^1,z')=(z^1,z^2,\ldots,z^n)$.

Given a holomorphic function $a_{-1}(z')\in \cO(\{0\}\times \Delta^{n-1})$ we define a current $T$ as follows. For a differentiable function $\varphi(z)$ with compact support on $\Delta^n$ we set
$$
T(\varphi)= \int_{B} a_{-1}(z')\frac{\pt \varphi(z)}{\pt z^1} dV_{n-1}
$$
where $dV_{n-1}= \frac{(\ii)^{n-1}}{2^{n-1}}dz^2\we \ol{dz}^2\we\ldots \we dz^n\we \ol{dz}^n$.
\begin{lemma}\label{le:extL1}
Any function $\chi \in L^1_{\mathrm{loc}}(\Delta^n,\R)$, whose restriction to \break  $\Delta^*\times \Delta^{n-1}$ is pluriharmonic can be written as
\begin{equation}\label{eq:extchi}
  \chi= \beta\log(|z^1|^2) + \mathrm{Re}(a_{-1}(z')/z^1) + \mathrm{Re}(F(z^1,z')),
\end{equation}
where $\beta\in \R$, $a_{-1}\in \cO(\{0\}\times \Delta^{n-1})$, and $F\in \cO(\Delta^n)$. In particular
\begin{equation}\label{eq:extddbchi}
  {\tiny\frac{\ii}{2\pi}\pt\ol\pt} \chi = \beta [B] + T,
\end{equation}
where $T$ is as above.
\end{lemma}
\begin{proof}
The first contribution in \eqref{eq:extchi} reflects the monodromy: For suitably chosen $\beta$ we have $\chi- \beta\log(|z^1|^2) = \mathrm{Re}(f)$, for some holomorphic function $f\in\cO(\Delta^*\times \Delta^{n-1})$. Since $f\in L^1_{\mathrm{loc}}(\Delta^*\times \Delta^{n-1})$ the principal part of $f$ is equal to $a_{-1}(z')/z^1$ so that $f(z)=a_{-1}(z')/z^1 +F(z)$. The term $\beta\log(|z^1|^2)$ yields the current of integration $\beta [B]$.

We now set $n=1$, $z=z^1$, and let $\varphi(z) \in C^\infty_0(\Delta,\R)$ be a function with compact support. Then
\begin{equation}\label{eq:cauchy1}
\frac{\pt \varphi}{\pt z}(0) = -\int_\Delta \frac{1}{\pi z} \frac{\pt^2\varphi}{\ol{\pt z}\pt z} \frac{\ii}{2}dz\we\ol{dz}.
\end{equation}
We interpret this equation as
$$
-\ddb\left(\frac{1}{z}\right)= 2\pi \frac{\pt}{\pt z}\Big|_{z=0}
$$
in the sense of currents.

For arbitrary $n$ let the letter $\varphi$  denote a real $(2n-2)$-form with compact support on $\Delta^n$. We consider the integral
\begin{equation}\label{eq:cauchy}
-{\frac{1}{2\pi}}\int_{\Delta^n} \frac{f(z)}{z^1} \ddb \varphi
\end{equation}
and apply the Fubini-Tonelli theorem. Write
$$
\varphi= \sum (-1)^{i+j}   \varphi_{i\ol{\jmath}}(z)\,dz^1\we\ldots\we \wh{dz^i}\we\ldots\we dz^n\we dz^{\ol 1} \we\ldots \we \wh{dz^\ol{\jmath}}\we\ldots\we dz^{\ol n}.
$$
Because of Stokes theorem the only  possibly non-vanishing contribution in \eqref{eq:cauchy} comes from the coefficient $\varphi_{i,\ol 1}$.
Now \eqref{eq:cauchy1} implies that \eqref{eq:cauchy} is equal to
$$
 c_n\int_{\Delta^{n-1}} a_{-1}(z')
\sum\frac{\pt \varphi_{i,\ol 1}}{\pt z^1}(0,z')dV_{n-1},
$$
with $c_n= 2^{n-1}(\ii)^{n-1} (-1)^{n(n-1)}$.
On the other hand, the contraction
\begin{equation}\label{eq:contr}
f(z)\frac{\pt}{\pt z^1} \,\llcorner\, \pt \varphi\Big|_{\{0\}\times \Delta^{n-1}}
\end{equation}
restricted to $\{0\}\times \Delta^{n-1}$ as a differential form equals
$$
c_n \sum a_{-1}(z') \frac{\pt \varphi_{i \ol 1}}{\pt z^i}(0,z') dV_{n-1}.
$$
Again the $(n,n-2)$- and $(n-2,n)$-components of $\varphi$ do not contribute to the integral along $\Delta^{n-1}$ of the term \eqref{eq:contr}.
\end{proof}
The lemma generalizes to the complement of a snc divisor: Let $\chi \in L^1_{\mathrm{loc}}(\Delta^n,\R)$ be a function, whose restriction to $(\Delta^*)^k\times \Delta^{n-k}=\{(z^1,\ldots,z^n)\}$ is pluriharmonic.
\begin{lemma}\label{le:extL1snc}
Let $\chi$ be as above. Then
\begin{equation}\label{eq:extL1snc}
  \chi= \sum_{j=1}^{k}\beta_j\log|z^j|^2 + \mathrm{Re}(\sum_{j=1}^{k} G_j/z^j  )
\end{equation}
  where $\beta_j\in \R$, and $G_j\in \cO(\Delta^n)$. In particular
  \begin{equation}
    {\ii}{\pt\ol\pt} \chi = \sum_j[B_j] + T
  \end{equation}
  where $[B_j]$ is the current of integration along $V(z^j)$, $j=1,\ldots,k$ and $T=\sum T_j$ with $T_j$ analogous to $T$ in \eqref{eq:extddbchi}.
\end{lemma}

\begin{definition}
Let $Z$ be a complex manifold an $B$ a snc divisor. Let $v$ be a section of the normal bundle $N_{B|Z}$, and denote by $\llcorner$ the contraction of a vector field and a differential form.

Then we define a current $T_v$ of dimension $(1,1)$ on $Z$ by
\begin{equation}\label{eq:T_v}
    T_v(\varphi)= S_v(d\varphi),
\end{equation}
where
$$
S_v(\psi)= \int_B \mathrm{Re}(v)\,\llcorner \, \psi
$$
for any $(2n-1)$-form with compact support.
\end{definition}

At this point, we assume for simplicity that $B$ consist of one component.

We want to study holomorphic line bundles $(\cL,h)$, equipped with a singular hermitian metric, whose restriction to $Z'=Z\backslash B$ is isomorphic to the trivial bundle $(\cO_{Z'},1)$ equipped with the trivial hermitian metric. We consider  a suitable covering of $Z$ by coordinate neighborhoods $(U_i,(z^1_i,\ldots,z^n_i))$ with $B\cap U_i=V(z^1_i)$ and $z'_i=(z^2_i,\ldots,z^n_i)$ and generating holomorphic sections $e_i$ of $\cL|U_i$. We set $h_i=\|e_i\|^2_{h}$, and $\chi_i= -\log h_i$. We will use the term of a flat $\mathbb Q$-line bundle on $Z$ for an element of $\mathrm{Pic}(Z)\otimes \mathbb Q$ such that an integer multiple is induced by an ordinary line bundle that possesses a flat structure.

\begin{theorem}\label{th:flat}
Let $(\cL,h)$ be a holomorphic line bundle on $Z$ equipped with a singular hermitian metric, whose restriction to $Z'=Z\backslash B$ is isomorphic to the trivial line bundle $(\cO_{Z'},1)$ carrying the trivial hermitian metric.

Then the line bundle $\cL$ on $Z$ carries a flat structure, or the following statements hold.
\begin{itemize}
\item[(i)]
As a $\mathbb Q$-line bundle the line bundle $\cL$ satisfies
$$
\cL\simeq \cO_Z(\beta\cdot B) \otimes L,
$$
where $\beta$ is a rational number, and $L$ a flat $\mathbb Q$-line bundle.
\item[(ii)]
We have
$$
\quad \chi_i = \beta \cdot \log(|z^1_i|^2) + 2 \mathrm{Re}\left(\frac{a_i(0,z'_i)}{z^1_i}\right),
$$
where
$$
v = a_i(0,z'_i) \frac{\pt}{\pt z^1_i}
$$
determines a holomorphic section of the normal bundle $N_{B|Z}$.
\item[(iii)]
The Chern current is equal to
\begin{equation}
  c_1(\cL,h)= \beta[B] + T_v.
\end{equation}
In particular, the closed current $T_v$ has support in $B$.
\end{itemize}
\end{theorem}
The corresponding statement holds, if $B$ is a snc divisor by Lemma~\ref{le:extL1snc}.

\begin{proof}
Let $g_{ij}$ be transition functions for $\cL$ on $U_{ij}=U_i\cap U_j$ such that $e_i = g_{ij} e_j $. By \eqref{eq:extchi}
$$
\chi_i|(U_i\backslash B) = \beta_i \log(|z_i|^2) + 2 \mathrm{Re}\left(\frac{f_i}{z_i}\right)
$$
for certain holomorphic functions $f_i\in \cO_{U_i}(U_i)$, $\beta_i\in \R$, and the coordinate functions $z_i=z^1_i$ with $V(z_i)= B\cap U_i$ can be identified with canonical sections of $\cO_Z(B)$ over $U_i$.

We observe that the {\em holomorphic} contribution of $f_i(z^1_i,z'_i)/z^1_i$ can be disregarded, as these only amount to coordinate transformations of the fiber coordinate of the line bundle $\cL$. So we can assume that $f_j(z^1_j, z'_j)= f_j(0,z'_j)$.

Let the transition functions $\gamma_{ij}$ define the line bundle $\cO_Z(B)$ with respect to $\{U_i\}$ so that $z_i= \gamma_{ij} z_j$.

The equation $h_i = |g_{ij}|^2 h_j$ implies
\begin{eqnarray*}
  \log|g_{ij}|^2 &=&\chi_j-\chi_i\\ &=&(\beta_j-\beta_i)\log|z_j|^2 + \beta_i\log|\gamma_{ji}|^2 + 2 \mathrm{Re}\left(\frac{f_j-\gamma_{ji}f_i}{z_j}   \right)\, , \\
  |g_{ij}|^2&=&|z_j|^{2(\beta_j-\beta_i)} \cdot |\gamma_{ji}|^{2\beta_i} \cdot \left| \exp{\frac{f_j- \gamma_{ji}f_i}{z_j}}  \right|^2  .
\end{eqnarray*}
Since $g_{ij}$ is holomorphic and non-zero at all points of $B\cap U_{ij}$, this equation implies
$$
\beta_i=\beta_j=:\beta  \quad\text{ and }\quad f_j=\gamma_{ji}f_i \, .
$$
In particular
$$
|g_{ij}|^2=|\gamma_{ji}|^{2\beta},
$$
and the $f_i/z_i$ determine a holomorphic section of $\cO_Z(B)$. Let the given singular hermitian metric on $\cO_Z(B)$ be denoted by $k$ with $k_i=k|U_i$. Then the above equation implies that $h_i  k_i^\beta = h_j k_j^\beta = \varphi$ is globally defined, and hence $c_{1,\R}(\cL)= -\beta \cdot c_{1,\R}(\cO_Z(B))\in H^2_\R(Z,\mathbb Z)$. This implies $\beta\in \mathbb Q$, unless $c_{1,\R}(\cO_Z(B))= c_{1,\R}(\cL)=0$.

We are left with the case $\beta\in \mathbb Q$. We choose holomorphic logarithms of $g_{ij}$ and $\gamma_{ij}$.  (These are not unique). We can write
$$
g_{ij}=\gamma_{ij}^\beta \; \eta_{ij},
$$
where the $\eta_{ij}$ are constants of modulus one. So far we do not have cocycles on the right hand side. Let $m\cdot \beta \in \mathbb Z$. Then the constants $\eta^m_{ij}$ define a cocycle with values in $U(1)$, i.e.\ a flat line bundle.
\end{proof}
The analogous statement holds for a snc divisor  $B$.

\section{Application to  Douady spaces}
Let $\mathbf D$ be a component of the Douady space of subvarieties of a projective variety $Z$ that has a non-empty intersection with the locus of smooth subvarieties. The Douady space is taken in the category of {\em reduced} complex spaces. However, from this point on we will restrict ourselves to the category of {\it normal} parameter spaces and explain later, why this is sufficient.

Let $\cX \hookrightarrow Z\times \DD \to \DD$ denote the universal family.  We consider a very ample line bundle on $Z$, equipped with the Fubini-Study metric and its restriction  $(\cL,h)$ to $\cX$. As above, we denote by $\mathbf{D'}\subset \DD$ the locus of smooth subvarieties. Let $A = \DD\backslash \mathbf{D'}$.

We know from Section~\ref{se:WPCurr} and Section~\ref{se:appBGS}  that the  \wp current, which we denote by  $\omega^{WP}_\DD$,  extends the Chern form of $(\lambda(\xi)|\mathbf{D'},h^Q_\xi)$, and that it has a continuous plurisubharmonic potential. For any base change map $g:R\to \DD$ with $g(R)\not\subset A$ the current $\omega^{WP}_R=g^*\omega^{WP}_{\DD}$ is well-defined, and it plays the role of a \wp current for the pulled-back family.

In particular the Chern form of $(\lambda(\xi)|\mathbf{D'},h^Q_\xi)$ extends to any desingularization of $\DD$ as a positive current, and the assumptions of Theorem~\ref{th:mainext} are satisfied. This yields the following fact.
\begin{proposition}\label{pr:mfdext}
After taking a modification $\tau:\mathbf{\wt D} \to \DD$ that is an isomorphism over $\mathbf{D'}$, the determinant line bundle $(\lambda(\xi), h^Q_\xi)|\mathbf{D'}$ extends to a holomorphic line bundle $\lambda_{\mathbf{\wt D}}$ on $\wt\DD$ together with a singular hermitian metric $h_{\wt\DD}$, whose curvature current is positive in the sense of currents. Also, we may assume that $B:=\tau^{-1}(A)$ is a snc divisor with components $B_i$.

In codimension two
\begin{equation}\label{eq:ext}
c_1(\lambda_{\wt \DD},h_{\wt\DD}) =  \omega^{WP}_{\wt\DD} + \sum_j \beta_j [B_j]
\end{equation}
holds, where $\beta_j\in \R$, and the $[B_j]$ denote currents of integration along the divisors $[B_j]$.
\end{proposition}
We observe that on the regular part of $\wt\DD$ the decomposition \eqref{eq:ext} is equal to Siu's decomposition (cf.\ \cite[(2.18)]{analmeth}),  in which the residual part $\omega^{WP}_{\wt\DD}$ has vanishing Lelong numbers, and where only finitely many currents of integration occur.

\begin{lemma}\label{le:Qext}
The Quillen metric $h^Q_\xi$ on $\lambda(\xi)$, which is defined over $\mathbf{D'}$, induces a singular hermitian metric on $\lambda(\xi)$ over $\DD$, i.e.\ $-\log h^Q_\xi$ is locally of class $L^1$ with respect to $\DD$.
\end{lemma}
We will use the same notation $h^Q_\xi$ for this metric over $\DD$.
\begin{proof}
We use the fact that $h^Q_\xi$ has positive curvature on $\mathbf{D'}$. Let $\wt\xi$ be the pullback of $\xi$ to $\cX\times_\DD \wt\DD$. In Section~\ref{se:appY} we studied restrictions to curves: By \eqref{eq:comphQ} and \eqref{eq:comphQg} the restriction of $h^Q_{\wt \xi}$ to a curve $C$ (not contained in $B$) is of the form $ |\sigma_C|^2  h^Q_{\mu^*(\wt \xi|C)}$, where $h^Q_{\mu^*(\wt \xi|C)}$ is a singular hermitian metric with a continuous local  potential for $-\log h^Q_{\mu^*(\wt \xi|C)}$, and $\sigma_C$ a canonical section of $\cO_C(B\cap C)$. (The map $\mu$ was constructed depending on $C$ in Section~\ref{se:appY}.)

Since we are dealing with projective algebraic varieties, the multiplicities of the functions $\sigma_C$ in \eqref{eq:comphQ} and \eqref{eq:comphQg} resp.\ are uniformly bounded. Now we can take a canonical section $\sigma_B$ for the snc divisor $B=\tau^{-1}(A)$ and a high multiple $M$ so that
$$
-\log(h^Q_{\wt\xi}/|\sigma_B|^{2M})|C
$$
is subharmonic and not identically equal to $-\infty$ for all curves $C$ not contained in $B$. Namely, we know that on such a curve $C$ the metric $h^Q_{\wt\xi}$ divided by some power (depending upon $C$) of the canonical section is continuous so that by taking $M$ large enough the potential $-\log(h^Q_{\wt\xi}/|\sigma_B|^{2M})$ tends to $-\infty$ when approaching the set $B$. Also this guarantees semi-continuity of the extension. We use the definition of plurisubharmonic functions on normal spaces by Grauert and Remmert from \cite{g-r}. Restricted to any local analytic curve intersecting $B$ the function is obviously subharmonic. In any case, the above potential extends as a \psh\ function to $B$ by assigning the value $-\infty$. Now $1/|\sigma_B|^{2M}$ is a singular hermitian metric so that $h^Q_{\wt\xi}$ is also a singular hermitian metric.

\end{proof}

{\bf Outline of the construction. } Over the space $\DD$ the determinant line bundle $\lambda(\xi)$ is defined by the flat universal family, and we know that the Quillen metric over $\DD'$ defines a singular metric on $\lambda(\xi)$ over all of $\DD$ by Lemma~\ref{le:Qext}. On the other hand there exists a holomorphic extension of $(\lambda^Q_\xi, h^Q_\xi)|\DD'$ in the sense of
Proposition~\ref{pr:mfdext}, which is based uponon the existence of the \wp current. The latter possesses a continuous potential by Proposition~\ref{pr:fibint}, and Theorem~\ref{th:VAthm2}. However the extension provided by the theorem is only unique up to an isomorphism in codimension one. Such isomorphisms are treated in Theorem~\ref{th:flat}. In this way extra terms occur, namely a flat line bundle $L$, and an exact current $S_v$, which is induced by a holomorphic section of the normal bundle of the critical divisor. The latter term will be eliminated by Yoshikawa's Theorem.

We keep the assumptions from Proposition~\ref{pr:mfdext} and Lemma~\ref{le:Qext}.

\begin{theorem}\label{th:main}
  There exists a desingularization $\mu:\wh \DD \to \wt\DD$ of the logarithmic pair $(\wt\DD,B)$ together with a flat holomorphic line bundle $L$ on $\wt \DD$ such that for an integer $m\in \mathbb N$ the following statements hold:
  \begin{itemize}
    \item[(i)]
    In codimension two we have
    \begin{equation}\label{eq:lambdaD}
      \lambda^{\otimes m}_{\wt\DD}= \tau^*\lambda^{\otimes m}(\xi) \otimes \cO_{\wt\DD}\left(\sum m \beta_j\cdot B_i\right)\otimes L,
    \end{equation}
    and $L$ is holomorphically trivial over the preimage of $\DD'$. On $\wh\DD$ we have
    $$
    \lambda^{\otimes m}_{\wh\DD}= \mu^*\tau^*\lambda^{\otimes m}(\xi) \otimes \cO_{\wh\DD}\left(\sum m \wh \beta_j\cdot \wh B_i\right)\otimes \wh L,
    $$
    where the $\wh B_i$ are the components of the snc divisor $\mu^*B$, $\wh \beta_i\in \R$, and $\wh L$ is a flat line bundle on $\wh\DD$.
    \item[(ii)]
    The Quillen metric $h^Q_\xi$ locally written as $e^{-\chi^Q_\xi}$ pulled back to $\wh\DD$ locally satisfies the  relation
    \begin{gather*}
    \chi^Q_\xi = \chi^{WP}+ \sum \wh\beta_i \log|z_i|^2,
    \end{gather*}
    where the continuous functions  $\chi^{WP}$ are local potentials for the \wp current, and the functions $z_i$ define canonical sections of $\cO_{\wh\DD}(\wh B_i)$.
    \item[(iii)] The Chern current of $(\lambda(\xi),h^Q_\xi)$ satisfies
    $$
    \omega^{WP}_{\wh\DD} = \mu^*\tau^* c_1(\lambda(\xi),h^Q_\xi)+ \sum \gamma_i[\wh B_i],\; \gamma_i\geq 0 .
    $$
  \end{itemize}
\end{theorem}
\begin{proof}
  The line bundles $\tau^*(\lambda(\xi))$, and $\lambda_{\wt\DD}$ exist on all of $\wt\DD$.
  We apply Theorem~\ref{th:flat} to $\wt\DD$ {\em in codimension two} first with numbers $\beta_i$ for each component of $B$ given by this theorem.  It follows from the construction that the line bundle $\cO_{\wt\DD}(\sum\beta_i\cdot B_i)$ is defined over all of $\wt\DD$, since the divisors $B_i$ exist as (Cartier) divisors on all of $\wt \DD$ (cf.\ Theorem~\ref{th:mainext} and remark thereafter). Now the line bundle $L$ can be defined by  \eqref{eq:lambdaD} on the whole space $\wt\DD$. So far it is holomorphic, and flat in codimension two. Theorem~\ref{th:flat} is applied to the pullbacks to $\wt\DD$ of the above bundles, which implies the second part of (i).

  In order to compute the Quillen metric $h^Q_\xi$ pulled back to $\wh\DD$ we apply Theorem~\ref{th:flat}(ii,iii) -- it is necessary to see that the sections $v$ and the related exact current $T_v$ vanish, which is a consequence of Yoshikawa's Theorem (Proposition~\ref{pr:cont}, and Proposition~\ref{pr:comphQ}). This is the last step.

  We restrict \eqref{eq:ext} to a curve $C$ that is not contained in $A$, and pull the data back to the proper transform $\wt C$ with respect to $\tau$. We consider the restriction $c_1(\lambda_{\wt \DD},h_{\wt\DD})|\wt C  = c_1(\lambda_{\wt \DD}|\wt C, h_{\wt\DD}|\wt C)$, which is meaningful since $C$ has a non-empty intersection with $\mathbf{D'}$. On the right hand side of this equation we have the restriction of the \wp current, which has a continuous potential, and furthermore we have the current of integration $\sum_j \beta*_j [B_j\cap \wt C]$. For the restriction of the \wp current to $\wt C$ we consider \eqref{eq:owpC} and \eqref{eq:wpchat} resp. These show that for the Quillen metric $h^Q_\xi$ on the determinant line bundle $\lambda(\xi)$
  $$
  -\log h^Q_\xi
  $$
  is continuous over such curves up to logarithmic poles, a fact that implies the absence of a current $S_v$.
\end{proof}

{\bf Non-normal structures.} Currents, and singular hermitian metrics on holomorphic line bundles on {\em non-normal}, reduced complex spaces are usually defined as currents on the normalization and singular hermitian metrics on the the respective line bundles pulled back to the normalization. By \cite[Theorem~2]{sch:arxiv} the extension theorem holds also in the non-normal case, as well as our main result.

{\bf Ampleness modulo boundary.} It would desirable to apply general results on restricted volumes of line bundles to Douady spaces, which in particular require that the \wp current $\omega^{WP}$ is a \ka current. At points with reduced fibers strict positivity was shown in \cite{a-s}. To include points with non-reduced fibers the weakly semistable reduction theorem of Viehweg \cite{vi} is necessary.

{\bf Acknowledgement.} The second named author would like to thank Ken-Ichi Yoshikawa for a helpful communication.

\end{document}